\documentclass[12pt]{amsart}

\usepackage{amssymb,latexsym}
\usepackage{bm}
\usepackage{mathrsfs}
\usepackage{graphicx}
\usepackage{psfrag}
\usepackage{wrapfig}
\usepackage{color}
\usepackage{enumitem}
\setlist[enumerate]{parsep=0pt plus 4pt,topsep=0pt plus 4pt}
\usepackage{url}
\usepackage{hyperref}
\usepackage{tikz}
\definecolor{darkblue}{RGB}{0,0,160}
\hypersetup{colorlinks,citecolor=black,filecolor=black,%
            linkcolor=darkblue,urlcolor=darkblue}
\allowdisplaybreaks

\newcommand{\excise}[1]{}

\newtheorem{thm}{Theorem}[section]
\newtheorem{lemma}[thm]{Lemma}

\newtheorem{cor}[thm]{Corollary}
\newtheorem{prop}[thm]{Proposition}

\theoremstyle{definition}
\newtheorem{example}[thm]{Example}
\newtheorem{remark}[thm]{Remark}
\newtheorem{defn}[thm]{Definition}

\numberwithin{equation}{section}

\newcommand{\Ring}[1]{\ensuremath{\mathbb{#1}}}

\newcommand\0{\mathbf{0}}

\newcommand\NN{\Ring{N}}

\newcommand\RR{{\mathbb R}}

\newcommand\ZZ{{\mathbb Z}}

\newcommand\ee{{\mathbf e}}

\newcommand\kk{\Bbbk}

\newcommand\xx{{\mathbf x}}

\newcommand\cA{{\mathcal A}}
\newcommand\cB{{\mathcal B}}

\newcommand\cN{N}

\newcommand\del{\partial}

\renewcommand\aa{{\mathbf a}}
\renewcommand\phi{\varphi}

\newcommand\into{\hookrightarrow}

\newcommand\onto{\twoheadrightarrow}

\newcommand\minus{\smallsetminus}

\newcommand\goesto{\rightsquigarrow}

\newcommand\nothing{\varnothing}

\renewcommand\iff{\Leftrightarrow}
\renewcommand\epsilon{\varepsilon}
\renewcommand\implies{\Rightarrow}

\newcommand\dx[1][]{\delta^{\hspace{.1ex}\xi}}

\newcommand\ol[1]{{\overline{#1}}}

\definecolor{lightred}{rgb}{1,.3,.3}

\newcommand\pur{\color{purple}}

\newcommand{\aoverb}[2]{{\genfrac{}{}{0pt}{1}{#1}{#2}}}
\def\twoline#1#2{\aoverb{\scriptstyle {#1}}{\scriptstyle {#2}}}

\DeclareMathOperator\Hom{Hom} 

\begin{document}

\title[Primary decomposition over partially ordered groups]
{\hspace{-9ex}Primary decomposition over partially ordered groups\hspace{-9ex}\mbox{}}
\author{Ezra Miller}
\address{Mathematics Department\\Duke University\\Durham, NC 27708}
\urladdr{\url{http://math.duke.edu/people/ezra-miller}}

\makeatletter
  \@namedef{subjclassname@2020}{\textup{2020} Mathematics Subject Classification}
\makeatother
%

\date{10 August 2020}

\begin{abstract}
Over any partially ordered abelian group whose positive cone is closed
in an appropriate sense and has finitely many faces, modules that
satisfy a weak finiteness condition admit finite primary
decompositions.  This conclusion rests on the introduction of basic
notions in the relevant generality, such as closedness of partially
ordered abelian groups, faces and their coprimary modules, and
finiteness conditions as well local and global support functors for
modules over partially ordered~groups.
\end{abstract}
\maketitle

\setcounter{tocdepth}{2}
\tableofcontents

\section{Introduction}\label{s:intro}

Primary decomposition yields concrete answers in combinatorial
commutative algebra, particularly in monomial \cite[Example~7.13]{cca}
and binomial \cite{eisenbud-sturmfels1996, mesoprimary, soccular}
contexts.  These answers come most naturally in the presence of a
multigrading that is positive \cite[\S8.1]{cca}.  When the grading is
by a torsion-free abelian group, positivity is equivalent to the group
being partially ordered: one element precedes another if their
difference lies in the positive cone of elements greater than~$0$.

What happens when the grading set isn't necessarily discrete?
Substantial parts of commutative algebra---%
especially homological algebra including the syzygy theorem
\cite{hom-alg-poset-mods}---generalize to modules over arbitrary
posets and have no need to rest on an underlying ring.  Alas, the part
of the theory relating to primary decomposition is not amenable to
arbitrary posets, because of a lack of natural prime ideals and
inability to localize.  Rather, the natural setting to carry out
primary decomposition is, in the best tradition of classical
mathematics \cite{birkhoff42, clifford40, riesz40}, over partially
ordered abelian groups (Definition~\ref{d:pogroup}).  Those provide an
optimally general context in which posets have a notion of ``face''
along which to localize---and to do so without altering the ambient
poset.  That is, a partially ordered group~$Q$ has an origin, namely
its identity~$\0$, and hence a positive cone~$Q_+$ of elements it
precedes.  A~\emph{face} of~$Q$ is a submonoid of $Q_+$ that is also a
downset therein (Definition~\ref{d:face}).  And as everything takes
place inside of the ambient group~$Q$, every localization of a
$Q$-module along a face (Definition~\ref{d:support}) remains a
$Q$-module.

Primary decomposition (Theorem~\ref{t:primDecomp}) expresses a given
module~$M$ as a submodule of a direct sum of coprimary modules
(Definition~\ref{d:coprimary} and
Theorem~\ref{t:elementary-coprimary}), each with an essential
submodule consisting of coprimary elements
(Definition~\ref{d:elementary-coprimary}).  Isolating all coprimary
elements functorially requires localization, after which local support
functors (Definition~\ref{d:local-support}) do the job, as in ordinary
commutative algebra and algebraic geometry.

\begin{example}\label{e:hyperbola-GD}
The downset $D$ in~$\RR^2$ consisting of all points beneath the upper
branch of the hyperbola $xy = 1$ canonically decomposes
(Theorem~\ref{t:PF}) as the union\vspace{-.4ex}%
$$%
\psfrag{x}{\tiny$x$}
\psfrag{y}{\tiny$y$}
  \begin{array}{@{}c@{}}\includegraphics[height=25mm]{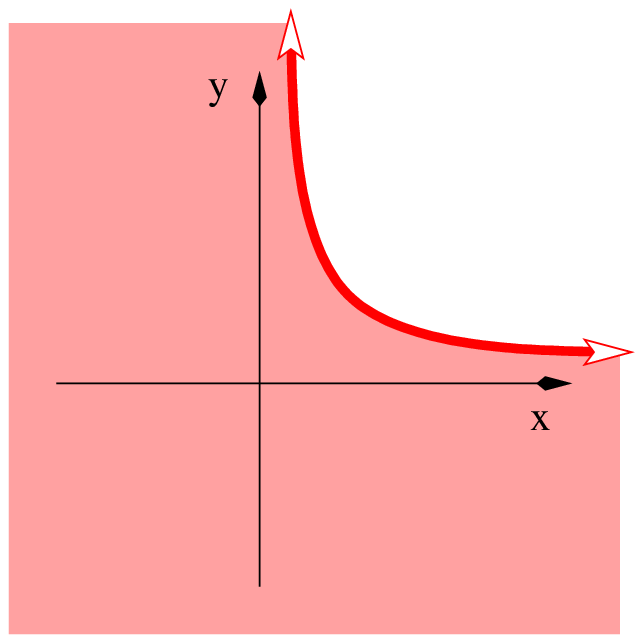}\end{array}
\ =\
  \begin{array}{@{}c@{}}\includegraphics[height=25mm]{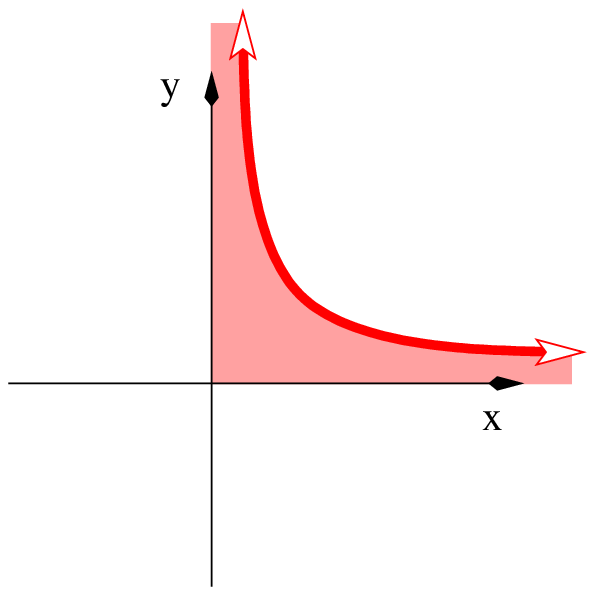}\end{array}
\cup\,
  \begin{array}{@{}c@{}}\includegraphics[height=25mm]{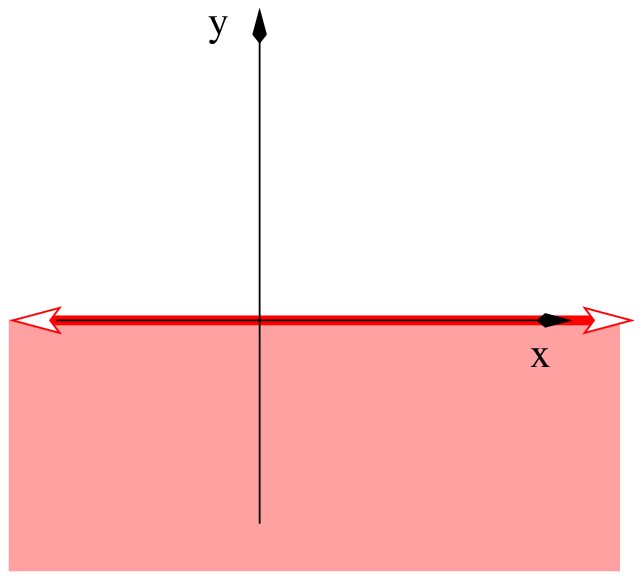}\end{array}
\cup\,
  \begin{array}{@{}c@{}}\includegraphics[height=25mm]{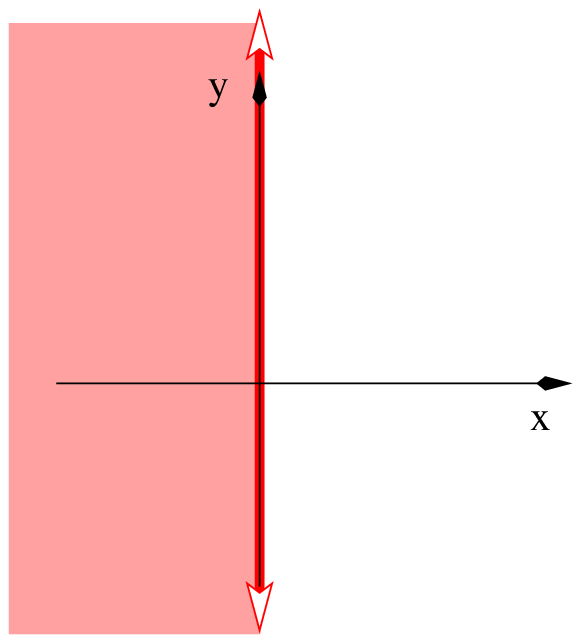}\end{array}
$$\vspace{-.6ex}%
of its subsets of coprimary elements of various types:
every red point in the
\begin{itemize}
\item%
leftmost subset on the right dies when pushed over to the right or up far enough;
\item%
middle subset dies in the \emph{localization} of~$D$ along the
$x$-axis (Definition~\ref{d:PF} or Definition~\ref{d:support}) when
pushed up far enough; and
\item%
rightmost subset dies locally along the $y$-axis when pushed over far
enough.
\end{itemize}
\end{example}

\begin{example}\label{e:hyperbola-PD}
The union in Example~\ref{e:hyperbola-GD} results in a canonical
primary decomposition
$$%
\psfrag{x}{\tiny$x$}
\psfrag{y}{\tiny$y$}
  \kk\!
  \left[
  \begin{array}{@{}c@{}}\includegraphics[height=25mm]{hyperbola}\end{array}
  \right]
\:\into\ 
  \kk\!
  \left[
  \begin{array}{@{}c@{}}\includegraphics[height=25mm]{hyperbola}\end{array}
  \right]
\oplus\,
  \kk\!
  \left[
  \begin{array}{@{}c@{}}\includegraphics[height=25mm]{x-component}\end{array}
  \right]
\oplus\,
  \kk\!
  \left[
  \begin{array}{@{}c@{}}\includegraphics[height=25mm]{y-component}\end{array}
  \right]
$$
of the downset module $\kk[D]$ over~$\RR^2$ (Corollary~\ref{c:PF}).
Elements in the lower-left quadrant locally die any type of death.
\end{example}

It bears emphasizing that primary decomposition of downset modules, or
equivalently, expressions of downsets as unions of coprimary downsets
(cogenerated by the $\tau$-coprimary elements for some face~$\tau$;
see Definition~\ref{d:primDecomp}), is canonical by Theorem~\ref{t:PF}
and Corollary~\ref{c:PF}, generalizing the canonical primary
decomposition of monomial ideals in ordinary polynomial rings.
However, notably lacking from primary decomposition theory over
arbitrary polyhedral partially ordered abelian groups is a notion of
minimality---alas, a lack that is intrinsic.

\begin{example}\label{e:minimality}
Although three death types occur in~$D$ in
Example~\ref{e:hyperbola-GD}, and hence in the union there, the final
two summands in the primary decomposition of~$\kk[D]$ in
Example~\ref{e:hyperbola-PD} are redundant.  One can, of course,
simply omit the redundant summands, but for arbitrary polyhedral
partially ordered groups no criterion is known for detecting a~priori
which summands should be omitted.
\end{example}

The failure of minimality here stems from geometry that can only occur
in partially ordered groups more general than finitely generated free
ones.  More specifically, although $D$ contains, for instance,
elements that die deaths of type ``$x$-axis'', the boundary of~$D$
fails to contain an actual translate of the face of~$\RR^2_+$ that is
the positive $x$-axis.  This can be seen as a certain failure of
localization to commute with taking homomorphisms into~$\kk[D]$
(Remark~\ref{r:commutes}); it is the source of much of the subtlety in
the theory developed in the sequel \cite{essential-real} to this
paper, whose purpose is partly to rectify, for $\RR^n$-modules
(equivalently, $\RR^n$-graded modules over the real-exponent
polynomial ring $\kk[\RR^n_+]$), the failure of minimality in
Example~\ref{e:hyperbola-PD}.

In general, the innovation in this paper is getting the hypotheses
right and selecting the appropriate one of the usually many equivalent
formulations for any given claim or proof.  The failure of
localization to commute with taking homomorphisms is a quintessential
feature of the more general theory that guides the sometimes
nonobvious choices required.

Substantial impetus for this work comes from applied topology, where
the focus is on modules over posets whose underlying partially ordered
groups are real vector spaces \cite{lesnick-interleav2015,
fruitFlyModuli, kashiwara-schapira2018, kashiwara-schapira2019,
strat-conical}; see \cite{qr-codes} for context (and an early draft of
this paper in \S3 there).  The view toward algorithmic computation
draws the focus to the case where $Q$ is \emph{polyhedral}, meaning
that it has only finitely many faces (Definition~\ref{d:face}).  This
notion is apparently new for arbitrary partially ordered abelian
groups.  Its role here is to guarantee finiteness of primary
decomposition of downset-finite modules (Theorem~\ref{t:primDecomp}).
To illuminate the meaning of this main result in the context of
persistent homology, multiparameter features can die in many ways,
persisting indefinitely as some of the parameters increase without
limit but dying when any of the others increase sufficiently.  In
persistence language, a single element in a module over a partially
ordered group can a~priori be mortal or immortal in more than one way.
But some elements die ``pure deaths'' of only a single type~$\tau$.
These are the $\tau$-coprimary elements for a face~$\tau$.  In the
concrete setting of a partially ordered real vector space with closed
positive cone, as in more general settings, a coprimary element is
characterized (Example~\ref{e:closed},
Definition~\ref{d:elementary-coprimary}, and
Theorem~\ref{t:elementary-coprimary})~as
\begin{enumerate}
\item%
\emph{$\tau$-persistent}: it lives when pushed up arbitrarily along
the face~$\tau$; and

\item%
\emph{$\ol\tau$-transient}: it eventually dies when pushed up in any
direction outside~of~$\tau$.
\end{enumerate}
Primary decomposition tells the fortune of every element: its death
types are teased apart as the ``pure death types'' of the coprimary
summands where the element lands with nonzero image.  In the ordinary
situation of one parameter, the only distinction being made here is
that a feature can be mortal or immortal.  Beyond the intrinsic
mathematical value, decomposing a module according to these
distinctions has concrete benefits for statistical analysis using
multipersistence \cite{primary-distance}.

\section{Polyhedral partially ordered groups}\label{s:polyhedral}

\begin{defn}\label{d:poset-module}
Let $Q$ be a partially ordered set (\emph{poset}) and~$\preceq$ its
partial order.  A \emph{module over~$Q$} (or a \emph{$Q$-module})
is
\begin{itemize}
\item%
a $Q$-graded vector space $M = \bigoplus_{q\in Q} M_q$ with
\item%
a homomorphism $M_q \to M_{q'}$ whenever $q \preceq q'$ in~$Q$
such that
\item%
$M_q \to M_{q''}$ equals the composite $M_q \to M_{q'} \to
M_{q''}$ whenever $q \preceq q' \preceq q''$.
\end{itemize}
A \emph{homomorphism} $M \to \cN$ of $Q$-modules is a
degree-preserving linear map, or equivalently a collection of vector
space homomorphisms $M_q \to \cN_q$, that commute with the structure
homomorphisms $M_q \to M_{q'}$ and $\cN_q \to \cN_{q'}$.
\end{defn}

The next definition, along with elementary foundations surrounding it,
can be found in Goodearl's book \cite[Chapter~1]{goodearl86}.

\begin{defn}\label{d:pogroup}
An abelian group~$Q$ is \emph{partially ordered} if it is generated
by a submonoid~$Q_+$, called the \emph{positive cone}, that has
trivial unit group.  The partial order is: $q \preceq q' \iff q' - q
\in Q_+$.  All partially ordered groups in this paper are assumed
abelian.
\end{defn}

\begin{example}\label{e:discrete-pogroup}
The finitely generated free abelian group $Q = \ZZ^n$ can be partially
ordered with any positive cone~$Q_+$, polyhedral or otherwise,
resulting in a \emph{discrete partially ordered group}.  The free
commutative monoid $Q_+ = \NN^n$ of integer vectors with nonnegative
coordinates is the most common instance and serves as a well behaved,
well known foundational case.  For notational clarity, $\ZZ_+^n$
always means the nonnegative orthant in~$\ZZ^n$, which induces the
standard componentwise partial order on~$\ZZ^n$.  Other partial orders
can be specified using notation $Q \cong \ZZ^n$ with arbitrary
positive cone~$Q_+$.
\end{example}

\begin{example}\label{e:real-pogroup}
The group $Q = \RR^n$ can be partially ordered with any positive
cone~$Q_+$, polyhedral or otherwise, closed, open (away from the
origin~$\0$) or anywhere in between, resulting in a \emph{real
partially ordered group}.  The orthant $Q_+ = \RR_+^n$ of vectors with
nonnegative coordinates is most useful for purposes such as
multipersistence.  For notational clarity, $\RR_+^n$ always means the
nonnegative orthant in~$\RR^n$, which induces the standard
componentwise partial order on~$\RR^n$.  Other partial orders can be
specified using notation $Q \cong \RR^n$ with arbitrary positive
cone~$Q_+$.
\end{example}

\begin{example}\label{e:torsion}
Definition~\ref{d:pogroup} allows the group to have torsion.  Thus the
submonoid of $\ZZ \oplus \ZZ/2\ZZ$ generated by $\bigl[\twoline
10\bigr]$ and $\bigl[\twoline 11\bigr]$
is a positive cone in the group.  There is a continuous version in
which the resulting partial order is easier to
\begin{figure}[ht]
\vspace{-1ex}
$$%
\includegraphics[height=20mm]{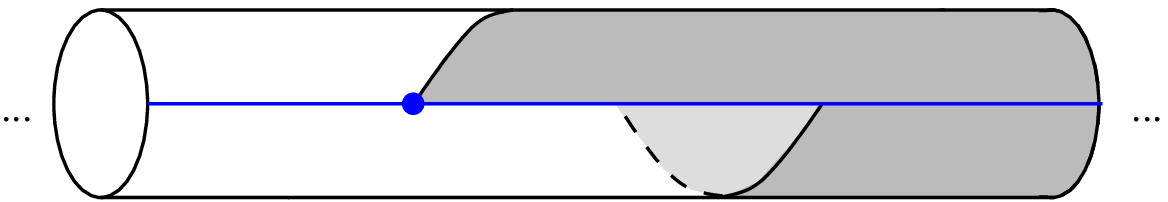}
$$
\vspace{-4ex}
\end{figure}
see geometrically: $Q = \RR \times \RR/\ZZ$ with $Q_+$ generated
by $\bigl[\twoline 10\bigr]$ and $\bigl[\twoline 11\bigr]$.  In the
figure, the blue center line is the first factor~$\RR$, with
origin~$\0$ at the fat blue dot.  The positive cone~$Q_+$ is shaded.
\end{example}

The following allows the free use of the language of either
$Q$-modules or $Q$-graded $\kk[Q_+]$-modules, as appropriate to
the context.

\begin{lemma}\label{l:Q-graded}
A module over a partially ordered abelian group~$Q$ is the same thing
as a $Q$-graded\/ module over the monoid algebra~$\kk[Q_+]$ of the
positive cone.\hfill\qed
\end{lemma}

\begin{example}\label{e:ZZn-graded}
When $Q = \ZZ^n$ and $Q_+ = \NN^n$, the relevant monoid algebra is the
polynomial ring $\kk[\NN^n] = \kk[\xx]$, where $\xx = x_1,\dots,x_n$
is a sequence of $n$ commuting~variables.  This is the classical case;
see \cite[\S8.1]{cca}, for example.
\end{example}

Primary decomposition of $Q$-modules depends on certain finiteness
conditions.  In ordinary commutative algebra, where $Q = \ZZ^n$, the
finiteness comes from~$Q_+$, which is assumed to be finitely generated
(so it is an \emph{affine semigroup}).  This condition implies that
finitely generated $Q$-modules are noetherian: every increasing chain
of submodules stabilizes.  Primary decomposition is then usually
derived as a special case of the theory for finitely generated modules
over noetherian rings.  But the noetherian condition is stronger than
necessary: it suffices for the positive cone to have finitely many
faces, in the following sense, along with the (much weaker)
``downset-finite'' replacement for the noetherian condition in
Section~\ref{s:prim-decomp}.  To the author's knowledge, the notion of
polyhedral partially ordered group is new, and there is no existing
literature on primary decomposition in this setting.

\begin{defn}\label{d:face}
A \emph{face} of the positive cone~$Q_+$ of a partially ordered
group~$Q$ is a submonoid $\sigma \subseteq Q_+$ such that
$Q_{+\!} \minus \sigma$ is an ideal of the monoid~$Q_+$.
Sometimes it is simpler to say that~$\sigma$ is a \emph{face}
of~$Q$.  Call~$Q$ \emph{polyhedral} if it has only finitely
many~faces.
\end{defn}

Polyhedrality suffices to prove existence of (finite) primary
decomposition (Theorem~\ref{t:primDecomp}) in the presence of weak
finiteness condition on the module being decomposed.  However, many of
the ingredients, such as localization along or taking support on a
face (Definition~\ref{d:PF}), make sense also under a different sort
of~\mbox{hypothesis}.

\begin{defn}\label{d:closed}
Let $Q$ be a partially ordered group.
\begin{enumerate}
\item%
A \emph{ray} of the positive cone~$Q_+$ is a face that is totally
ordered as a partially ordered submonoid of~$Q$.

\item%
The partially ordered group $Q$ is \emph{closed} if the complement
$Q_{+\!} \minus \tau$ of each face~$\tau$ is generated as an upset
(i.e., as an ideal) of~$Q_{+\!}$ by $\rho \minus \{\0\}$ for the
rays $\rho \not\subseteq \tau$.
\end{enumerate}
\end{defn}

\begin{example}\label{e:closed}
Any real partially ordered group (Example~\ref{e:real-pogroup})~$Q$
whose positive cone~$Q_+$ is closed in the usual topology on~$Q$
is a closed partially ordered group by the Krein--Milman theorem:
$Q_+$ is the set of nonnegative real linear combinations of vectors
on extreme rays of~$Q_+$.  For instance, a non-polyhedral closed
partial order on $Q = \RR^3$ results by taking $Q_+$ to be a cone
over a~disk, such as either half of the cone $x^2 + y^2 \leq z^2$.  In
contrast, if~$Q_+$ is an intersection of finitely many closed
half-spaces, then there are only finitely many extreme rays.  (This
case is crucial in applications---see \cite{kashiwara-schapira2019,
essential-real}, for instance.)  Even in the polyhedral case the cone
need not be rational; that is, the vectors that generate it---or the
linear functions defining the closed half-spaces whose intersection
is~$Q_+$---need not have rational entries.
\end{example}

\begin{example}\label{e:discrete-polyhedral}
Any discrete partially ordered group
(Example~\ref{e:discrete-pogroup}) whose positive cone is a finitely
generated submonoid is automatically both polyhedral and closed
\cite[Lemma~7.12]{cca}.  A discrete partially ordered group can also
have a positive cone that is not a finitely generated submonoid, such
as $Q = \ZZ^2$ with $Q_+ = C \cap \ZZ^2$ for the cone $C \subseteq
\RR^2$ generated by $\bigl[\twoline 10\bigr]$ and $\bigl[\twoline
1\pi\bigr]$.  This particular irrational cone yields a partially
ordered group that is polyhedral but not closed.  Indeed, there are
fewer than the expected faces, because only some of the faces of~$C$
result in faces of~$Q_+$ itself.  The image of~$Q$ is not discrete in
the quotient of $Q \otimes \RR$ modulo the subgroup spanned by the
irrational real face, which can have unexpected consequences for the
algebra of poset modules under localization along such a face.
\end{example}

\begin{example}\label{e:torsion'}
The cylindrical group~$Q$ in Example~\ref{e:torsion} has two faces:
the origin~$\0$ (the fat blue dot) and~$\RR_{+\!}$ (the rightmost half
of the horizontal blue center line).
\end{example}

\section{Primary decomposition of downsets}\label{s:downsets}

\begin{defn}\label{d:indicator}
Fix a poset~$Q$.  The vector space $\kk[Q] = \bigoplus_{q\in Q} \kk$
that assigns $\kk$ to every point of~$Q$ is a $Q$-module with identity
maps on~$\kk$.  More generally,
\begin{enumerate}
\item\label{i:upset}%
an \emph{upset} (also called a \emph{dual order ideal}) $U \subseteq
Q$, meaning a subset closed \mbox{under} going upward in~$Q$ (so $U +
Q_+ = U$, when $Q$ is a partially ordered group) determines an
\emph{indicator submodule} or \emph{upset module} $\kk[U] \subseteq
\kk[Q]$; and
\item\label{i:downset}%
dually, a \emph{downset} (also called an \emph{order ideal}) $D
\subseteq Q$, meaning a subset closed under going downward in~$Q$ (so
$D - Q_+ = D$, when $Q$ is a partially ordered group) determines an
\emph{indicator quotient module} or \emph{downset module} $\kk[Q]
\onto \kk[D]$.
\end{enumerate}
\end{defn}

\begin{defn}\label{d:PF}
Fix a face~$\tau$ of the positive cone~$Q_+$ in a polyhedral or
closed partially ordered group~$Q$ and a downset $D \subseteq Q$.
Write $\ZZ \tau$ for the subgroup of~$Q$ generated~by~$\tau$.
\begin{enumerate}
\item\label{i:localization}%
The \emph{localization} of~$D$ \emph{along~$\tau$} is the subset
$$%
  D_\tau = \{q \in D \mid q + \tau \subseteq D\}.
$$

\item\label{i:globally-supported}%
An element $q \in D$ is \emph{globally supported on~$\tau$} if $q
\not\in D_{\tau'}$ whenever $\tau' \not\subseteq \tau$.

\item\label{i:on-tau}%
The part of~$D$ \emph{globally supported on~$\tau$} is
$$%
  \Gamma_{\!\tau} D = \{q \in D \mid q \text{ is globally supported on }\tau\}.
$$

\item%
An element $q \in D$ is \emph{locally supported on~$\tau$} if $q$ is
globally supported on~$\tau$ in~$D_\tau$.

\item%
The \emph{local $\tau$-support} of~$D$ is the subset
$\Gamma_{\!\tau}(D_\tau) \subseteq D$ consisting of elements globally
supported on~$\tau$ in the localization~$D_\tau$.

\item\label{i:primary-component}%
The \emph{$\tau$-primary component} of~$D$ is the downset
$$%
  P_\tau(D) = \Gamma_{\!\tau}(D_\tau) - Q_+
$$
cogenerated by the local $\tau$-support of~$D$.
\end{enumerate}
\end{defn}

\begin{example}\label{e:PF}
The local $\tau$-supports of the under-hyperbola downset in
Example~\ref{e:hyperbola-GD} are the subsets depicted on the
right-hand side there, for the faces $\tau = \0$, $x$-axis, and
$y$-axis, respectively.  The corresponding primary components are
depicted in Example~\ref{e:hyperbola-PD}.  In contrast, the global
support on (say) the $y$-axis consists of the part of the local
support that sits strictly above the $x$-axis, and the global support
at~$\0$ is the part of~$D$ strictly in the positive quadrant.

This example demonstrates that the $\tau$-primary component of~$D$ in
Definition~\ref{d:PF} need not be supported on~$\tau$.  Indeed, $D =
P_\0(D)$ here, and points outside of~$Q_+$ are not supported at the
origin, being instead locally supported at either the $x$-axis (if the
point is below the $x$-axis) or the $y$-axis (if the point is behind
the $y$-axis).
\end{example}

\begin{remark}\label{r:PF}
Definition~\ref{d:PF} makes formal sense in any partially ordered
group, but extreme
caution is recommended without the closed or polyhedral assumptions.
Indeed, without such assumptions, faces can be virtually present, such
as the irrational face in Example~\ref{e:discrete-polyhedral} or a
missing face in a real polyhedron that is not closed.  In such cases,
aspects of Definition~\ref{d:PF} might produce unintended output.  The
natural generality for the concepts in
Definition~\ref{d:PF}~is~unclear.
\end{remark}

\begin{example}\label{e:coprincipal}
The \emph{coprincipal} downset $\aa + \tau - Q_+$ inside of $Q =
\ZZ^n$ \emph{cogenerated} by~$\aa$ \emph{along~$\tau$} is globally
supported along~$\tau$.  It also equals its own localization
along~$\tau$, so it equals its local $\tau$-support and is its own
$\tau$-primary component.  Note that when $Q_+ = \NN^n$, faces
of~$Q_+$ correspond to subsets of~$[n] = \{1,\dots,n\}$, the
correspondence being $\tau \leftrightarrow \chi(\tau)$, where
$\chi(\tau) = \{i \in [n] \mid \ee_i \in \tau\}$ is the
\emph{characteristic subset} of~$\tau$ in~$[n]$.  (The vector~$\ee_i$
is the standard basis vector whose only nonzero entry is $1$ in
slot~$i$.)
\end{example}

\begin{remark}\label{r:freely}
The localization of~$D$ along~$\tau$ is acted on freely by~$\tau$.
Indeed, $D_\tau$ is the union of those cosets of~$\ZZ \tau$ each of
which is already contained in~$D$.  The minor point being made here is
that the coset $q + \ZZ \tau$ is entirely contained in~$D$ as soon as
$q + \tau \subseteq D$ because $D$ is a downset: $q + \ZZ \tau = q +
\tau - \tau \subseteq q + \tau - Q_+ \subseteq D$ if~\mbox{$q + \tau
\subseteq D$}.
\end{remark}

\begin{remark}\label{r:monomial-localization}
The localization of~$D$ is defined to reflect localization at the
level of $Q$-modules: enforcing invertibility of structure
homomorphisms $\kk[D]_q \to \kk[D]_{q+f}$ for $f \in \tau$ results in
a localized indicator module $\kk[D][\ZZ \tau] = \kk[D_\tau]$; see
Definition~\ref{d:support}.
\end{remark}

\begin{example}\label{e:support}
Fix a downset $D$ in a partially ordered group~$Q$ that is closed
(Definition~\ref{d:closed} and subsequent examples).  An element $q
\in D$ is globally supported on~$\tau$ if and only if it lands outside
of~$D$ when pushed far enough up in any direction outside
of~$\tau$---that is, every $f \in Q_{+\!} \minus \tau$ has a
nonnegative integer multiple~$\lambda f$~with~$\lambda f +\nolinebreak
q \not\in D$.

One implication is easy: if every $f \in Q_{+\!} \minus \tau$ has
$\lambda f + q \not\in D$ for some $\lambda \in \NN$, then any element
$f' \in \tau' \minus \tau$ has a multiple $\lambda f' \in \tau'$ such
that $\lambda f' + q \not\in D$, so $q \not\in D_{\tau'}$.  For the
other direction, use Definition~\ref{d:closed}: $q \in \Gamma_{\!\tau}
D \implies q \not\in D_\rho$ for all rays~$\rho$ of~$Q_{+\!}$ that are
not contained in~$\tau$, so along each such ray~$\rho$ there is a
group element~$v_\rho$ with $v_\rho + q \not\in D$.  Given $f \in
Q_{+\!}  \minus \tau$, choose $\lambda \in \NN$ big enough so that
$\lambda f \succeq v_\rho$ for some~$\rho$.  This argument is the
purpose of the closed hypothesis; see the proof of
Theorem~\ref{t:elementary-coprimary}.
\end{example}

\begin{defn}\label{d:primDecomp}
Fix a downset~$D$ in a polyhedral partially ordered group~$Q$.
\begin{enumerate}
\item\label{i:coprimary}%
The downset~$D$ is \emph{coprimary} if $D = P_\tau(D)$ for some
face~$\tau$ of the positive cone~$Q_+$.  If $\tau$ needs to
specified then $D$ is called \emph{$\tau$-coprimary}.

\item%
A \emph{primary decomposition} of~$D$ is an expression $D =
\bigcup_{i=1}^r D_i$ of coprimary downsets~$D_i$, called
\emph{components} of the decomposition.
\end{enumerate}
\end{defn}

\begin{thm}\label{t:PF}
Every downset $D$ in a polyhedral partially ordered group~$Q$ is the
union $\bigcup_\tau \Gamma_{\!\tau}(D_\tau)$ of its local
$\tau$-supports for all faces $\tau$ of the positive cone.
\end{thm}
\begin{proof}
Given an element $q \in D$, finiteness of the number of faces implies
the existence of a face~$\tau$ that is maximal among those such that
$q \in D_\tau$; note that $q \in D = D_\0$ for the trivial face $\0$
consisting of only the identity of~$Q$.  It follows immediately that
$q$ is supported on~$\tau$ in~$D_\tau$.
\end{proof}

\begin{cor}\label{c:PF}
Every downset $D$ in a polyhedral partially ordered group~$Q$ has a
canonical primary decomposition $D = \bigcup_\tau P_\tau(D)$, the
union being over all faces~$\tau$ of the positive cone with nonempty
support $\Gamma_{\!\tau}(D_\tau)$.
\end{cor}

\begin{remark}\label{r:disjoint}
The union in Theorem~\ref{t:PF} is not necessarily disjoint.  Nor,
consequently, is the union in Corollary~\ref{c:PF}.  There is a
related union, however, that is disjoint: the sets $(\Gamma_{\!\tau}
D) \cap D_\tau$ do not overlap.  But their union need not be all
of~$D$; try Example~\ref{e:PF}, where the negative quadrant intersects
none of the sets~$(\Gamma_{\!\tau} D) \cap D_\tau$.

Algebraically, $(\Gamma_{\!\tau} D) \cap D_\tau$ should be interpreted
as taking the elements of~$D$ globally supported on~$\tau$ and then
taking their images in the localization along~$\tau$, which deletes
the elements that aren't locally supported on~$\tau$.  That is,
$(\Gamma_{\!\tau} D) \cap D_\tau$ is the set of graded degrees
from~$Q$ where the image of $\Gamma_{\!\tau}\kk[D] \to \kk[D]_\tau$ is
nonzero.
\end{remark}

\begin{example}\label{e:PF'}
The decomposition in Theorem~\ref{t:PF}---and hence
Corollary~\ref{c:PF}---is not necessarily minimal: it might be that
some of the canonically defined components can be omitted.  This
occurs, for instance, in Example~\ref{e:hyperbola-PD}.  The general
phenomenon, as in this hyperbola example, stems from geometry of the
elements in~$D_\tau$ supported on~$\tau$, which need not be bounded in
any sense, even in the quotient $Q/\ZZ \tau$.  In contrast, for
(say) quotients by monomial ideals in the polynomial
ring~$\kk[\NN^n]$, only finitely many elements have support at the
origin, and the downset they cogenerate is consequently~artinian.
\end{example}

\section{Localization and support}\label{s:local-support}

\begin{defn}\label{d:support}
Fix a face~$\tau$ of a partially ordered group~$Q$.  The
\emph{localization} of a $Q$-module~$M$ \emph{along~$\tau$} is the
tensor product
$$%
  M_\tau = M \otimes_{\kk[Q_{+\!}]} \kk[Q_{+\!} + \ZZ \tau],
$$
viewing~$M$ as a $Q$-graded $\kk[Q_{+\!}]$-module.  The
submodule of~$M\hspace{-1.17pt}$ \emph{globally
supported~on~$\tau$}~is
$$%
  \Gamma_{\!\tau} M
  =
  \bigcap_{\tau' \not\subseteq \tau}\bigl(\ker(M \to M_{\tau'})\bigr)
  =
  \ker \bigl(M \to \prod_{\tau' \not\subseteq \tau} M_{\tau'}\bigr).
$$
\end{defn}

\begin{example}\label{e:Gamma}
Definition~\ref{d:PF}.\ref{i:globally-supported} says that $1_q \in
\kk[D]_q = \kk$ lies in~$\Gamma_{\!\tau} \kk[D]$ if and only if $q \in
\Gamma_{\!\tau} D$, because $q \not\in D_{\tau'}$ if and only if $1_q
\mapsto 0$ under localization of~$\kk[D]$ along~$\tau'$.
\end{example}

\begin{example}\label{e:global-support}
The global supports of the indicator subquotient for the interval
$$%
\psfrag{x}{\tiny$x$}
\psfrag{y}{\tiny$y$}
  \begin{array}{@{}c@{}l@{}}
	{\pur .}\qquad\quad\ \\[-1.7ex]
	{\pur .}\qquad\quad\ \\[-1.7ex]
	{\pur .}\qquad\quad\ \\[-1ex]
	\includegraphics[height=29mm]{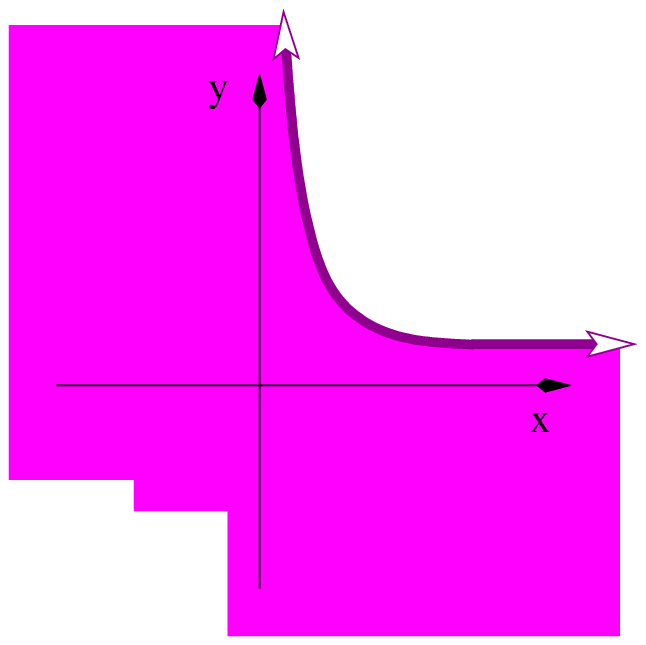}
	&\raisebox{3ex}{\pur$\!\cdot\!\cdot\!\cdot$}
	\end{array}\!\!
\quad\ \ \goesto\quad\ \
  \begin{array}{@{}c@{}}
	{\pur .}\quad\,\\[-1.7ex]
	{\pur .}\quad\,\\[-1.7ex]
	\makebox[0pt][l]{\quad$\tau = \nothing$}
	{\pur .}\quad\,\\[-1ex]
	\includegraphics[height=29mm]{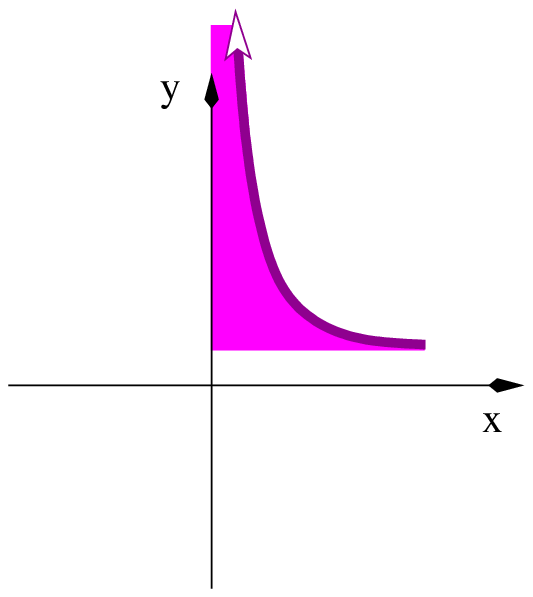}
	\end{array}
\ ,\,\quad
  \begin{array}{@{}c@{}l@{}}
	\phantom{.}\qquad\quad\ \ \\[-1.7ex]
	\phantom{.}\qquad\quad\ \ \\[-1.7ex]
	\makebox[0pt][l]{\qquad$\tau = \{x\}$}
	\phantom{.}\qquad\quad\ \ \\[-1ex]
	\includegraphics[height=29mm]{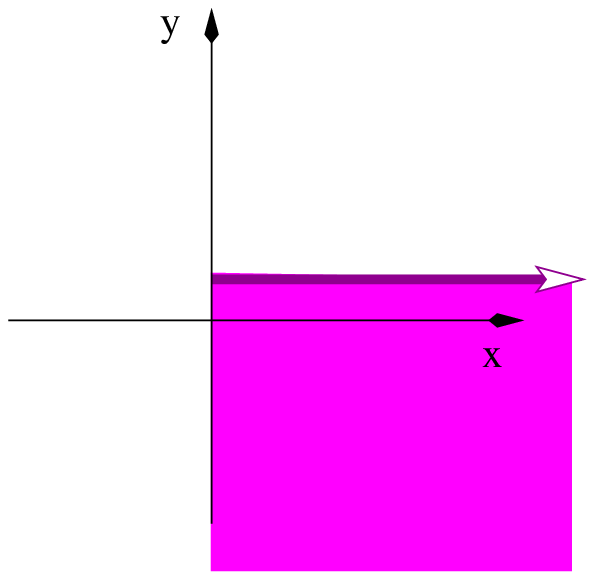}
	&\raisebox{3ex}{\pur$\!\cdot\!\cdot\!\cdot$}
	\end{array}
\ ,\,\quad
  \begin{array}{@{}c@{}}
	{\pur .}\qquad\quad\ \ \\[-1.7ex]
	{\pur .}\qquad\quad\ \ \\[-1.7ex]
	\makebox[0pt][l]{\qquad$\tau = \{y\}$}
	{\pur .}\qquad\quad\ \ \\[-1ex]
	\includegraphics[height=29mm]{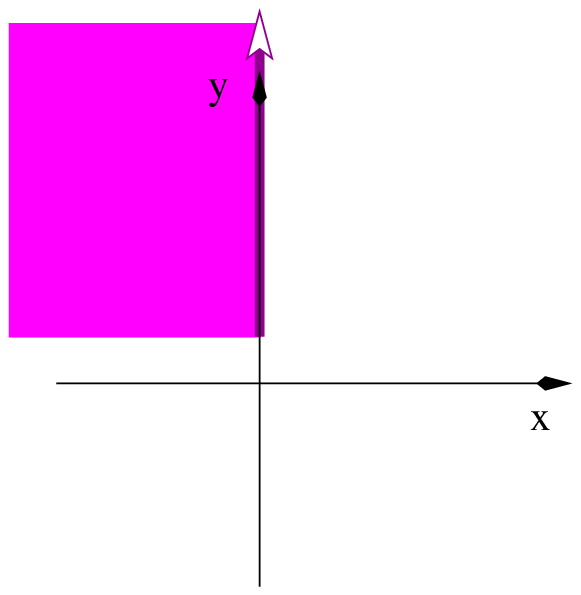}
	\end{array}
$$
in~$\RR^2$ on the left-hand side of this display are the indicator
subquotients for the intervals on the right-hand side, each labeled by
the relevant face~$\tau$.  Caution: this example is not to be confused
with Examples~\ref{e:hyperbola-GD}, \ref{e:hyperbola-PD}, \ref{e:PF},
and~\ref{e:PF'}, where the curve is a hyperbola whose asymptotes are
the two axes.  In contrast, here the upper boundary of the interval
has the vertical axis as an asymptote, whereas the horizontal axis is
exactly parallel to the positive end of the upper boundary.
\end{example}

\begin{lemma}\label{l:left-exact}
The kernel of any natural transformation between two exact covariant
functors is left-exact.  In more detail, if $\alpha$ and $\beta$ are
two exact covariant functors $\cA \to \cB$ for abelian categories
$\cA$ and~$\cB$, and $\gamma_X: \alpha(X) \to \beta(X)$ naturally for
all objects~$X$ of~$\cA$, then the association $X \mapsto \ker
\gamma_X$ is a left-exact covariant functor~$\cA \to \cB$.
\end{lemma}
\begin{proof}
This can be checked by diagram chase or spectral sequence.
\end{proof}

\begin{prop}\label{p:support-left-exact}
The global support functor\/ $\Gamma_{\!\tau\!}$ is left-exact.
\end{prop}
\begin{proof}
Use Lemma~\ref{l:left-exact}: global support is the kernel of the
natural transformation from the identity to a direct product of
localizations.
\end{proof}

\begin{prop}\label{p:support-localizes}
For modules over a polyhedral partially ordered group, localization
commutes with taking support: $(\Gamma_{\!\tau'} M)_\tau =
\Gamma_{\!\tau'}(M_\tau)$, and both sides are~$0$ unless~$\tau'
\supseteq \tau$.
\end{prop}
\begin{proof}
Localization along~$\tau$ is exact, so
$$%
  \ker(M \to M_{\tau''})_\tau
  =
  \ker\bigl(M_\tau \to (M_{\tau''})_\tau\bigr)
  =
  \ker\bigl(M_\tau \to (M_\tau)_{\tau''}\bigr).
$$
Since localization along~$\tau$ commutes with finite intersections of
submodules, $(\Gamma_{\!\tau'} M)_\tau$ is the intersection of the
leftmost of these modules over the faces $\tau'' \not\subseteq \tau'$,
of which there are only finitely many by the polyhedral hypothesis.
But $\Gamma_{\!\tau'}(M_\tau)$ equals the same intersection of the
rightmost of these modules by definition.  And if $\tau' \not\supseteq
\tau$ then one of these $\tau''$ equals~$\tau$, so $M_\tau \to
(M_\tau)_{\tau''} = M_\tau$ is the identity map, whose
kernel~is~$0$.
\end{proof}

\begin{remark}\label{r:commutes}
It is miraculous that localization commutes with taking support over
general polyhedral partially ordered groups, because localization does
not commute with taking relevant Hom functors in this setting.
Indeed, this commutativity failure occurs even when the source module
is a quotient $\kk[D] = \kk[Q_+ \minus U] = \kk[Q_+]/I$ modulo a
graded ideal $I = \kk[U]$ of the monoid algebra~$\kk[Q_+]$ of the
positive cone; see \cite[Remark~4.22]{essential-real}
for an explanation with examples.  The problem comes down to the
homogeneous prime ideals of the monoid algebra $\kk[Q_+]$ not being
finitely generated, so the quotient $\kk[\tau]$ fails to be finitely
presented.
However, taking the colimit of Hom functors of the form
$\Hom(\kk[D],M)$ for $\tau$-coprimary downsets~$D \subseteq Q_+$
eliminates the failure of commutativity.
\end{remark}

\begin{defn}\label{d:local-support}
Fix a $Q$-module $M$ for a polyhedral partially ordered
group~$Q$.  The \emph{local $\tau$-support} of~$M$ is the module
$\Gamma_{\!\tau} M_\tau$ of elements globally supported on~$\tau$ in
the localization~$M_\tau$, or equivalently (by
Proposition~\ref{p:support-localizes}) the localization along~$\tau$
of the submodule of~$M$ globally supported on~$\tau$.
\end{defn}

\begin{defn}\label{d:coprimary}%
A module $M$ over a polyhedral partially ordered group is
\emph{coprimary} if for some face~$\tau$, the localization map $M
\into M_\tau$ is injective and $\Gamma_{\!\tau} M_\tau$ is an
essential submodule of~$M_\tau$, meaning every nonzero submodule
of~$M_\tau$ intersects $\Gamma_{\!\tau}
M_\tau$~\mbox{nontrivially}.
\end{defn}

\begin{remark}\label{r:unique-face}
It is easy to check that over any polyhedral partially ordered group,
if a module $E$ is coprimary then it is $\tau$-coprimary for a unique
face~$\tau$ of~$Q$.
\end{remark}

The coprimary concept has an elementary, intuitive formulation in the
language of persistence, when the ambient partially ordered group is
polyhedral and closed.

\begin{defn}\label{d:elementary-coprimary}
Fix a face~$\tau$ of the positive cone~$Q_+$ in a partially ordered
group~$Q$.  A homogeneous element $y \in M_q$ in a $Q$-module~$M$ is
\begin{enumerate}
\item%
\emph{$\tau$-persistent} if it has nonzero image in $M_{q'}$ for all
$q' \in q + \tau$;

\item%
\emph{$\ol\tau$-transient} if, for each $f \in Q_+ \minus \tau$, the
image of~$y$ vanishes in $M_{q'}$ whenever $q' = q + \lambda f$ for
$\lambda \gg 0$;

\item%
\emph{$\tau$-coprimary} if it is $\tau$-persistent and
$\ol\tau$-transient.
\end{enumerate}
\end{defn}

\begin{remark}\label{r:ZZn-coprimary}
It is an interesting exercise to check that every element of a
coprimary module is coprimary when the polyhedral partially ordered
group is discrete (Example~\ref{e:discrete-polyhedral}) and closed
(Definition~\ref{d:closed}).
\end{remark}

\begin{thm}\label{t:elementary-coprimary}
Fix a $Q$-module~$M$ and a face~$\tau$ of the positive
cone~$Q_+$ in a closed polyhedral partially ordered group~$Q$.
The module~$M$ is $\tau$-coprimary if and only if every homogeneous
element divides a $\tau$-coprimary element, where $y \in M_q$
\emph{divides} $y' \in M_{q'}$ if $q \preceq q'$ and $y$ has
image~$y'$ under the structure morphism $M_q \to M_{q'}$.
\end{thm}
\begin{proof}
If $M$ is $\tau$-coprimary and $y \in M_q$ is a nonzero
homogeneous element, then $y$ is $\tau$-persistent because $M$ is a
submodule of~$M_\tau$ on which $\kk[\ZZ\tau]$ acts freely.  On the
other hand, $y$ divides a $\ol\tau$-transient element because
$\Gamma_{\!\tau} M_\tau$ is an essential submodule of~$M_\tau$:
the submodule of $M_\tau$ generated by~$y$ intersects
$\Gamma_{\!\tau} M_\tau$ nontrivially.
The closed hypothesis on~$Q$ implies that an element supported
on~$\tau$ is $\ol\tau$-transient, as in Example~\ref{e:support}.

The other direction does not require the closed hypothesis.  Assume
that every homogeneous element of~$M$ divides a $\tau$-coprimary
element.  The graded component of the localization~$M_\tau$ in
degree $q \in Q$ is the direct limit of~$M_q'$ over $q' \in q +
\tau$.  If $y \in M_q$ lies in $\ker(M \to M_\tau)$, then the
image of~$y$ must vanish in some~$M_{q'}$, whence $y = 0$ to begin
with, by $\tau$-persistence.  On the other hand, that $\Gamma_{\!\tau}
M_\tau$ is an essential submodule of~$M_\tau$ follows because
every $\ol\tau$-transient element is supported on~$\tau$.
\end{proof}


\section{Primary decomposition of modules}\label{s:prim-decomp}

\begin{defn}\label{d:primDecomp'}
Fix a $Q$-module $M$ over a polyhedral partially ordered group~$Q$.
A~\emph{primary decomposition} of~$M$ is an injection $M \into
\bigoplus_{i=1}^r M/M_i$ into a direct sum of coprimary quotients
$M/M_i$, called \emph{components} of the decomposition.
\end{defn}

\begin{remark}\label{r:primary}
Primary decomposition is usually phrased in terms of \emph{primary
submodules} $M_i \subseteq M$, which by definition have coprimary
quotients~$M/M_i$, satisfying $\bigcap_{i=1}^r M_i = 0$
in~$M$.  This is equivalent to Definition~\ref{d:primDecomp'}.
\end{remark}

\begin{example}\label{e:prim-decomp-downset}
By Theorem~\ref{t:elementary-coprimary}, a primary decomposition $D =
\bigcup_{i=1}^r \hspace{-1pt}D_i$ of a downset~$D$ yields a primary
decomposition of the corresponding indicator quotient, namely $\kk[D]
\into \bigoplus_{i=1}^r \kk[D_i]$ induced by the surjections $\kk[D]
\onto \kk[D_i]$.  See, Example~\ref{e:hyperbola-PD}, for instance.
\end{example}

\begin{example}\label{e:global-support'}
The interval module in Example~\ref{e:global-support} has a primary
decomposition
$$%
\psfrag{x}{\tiny$x$}
\psfrag{y}{\tiny$y$}
  \kk\!
  \left[
  \begin{array}{@{\!}c@{\!\!}}\includegraphics[height=29mm]{decomp}\end{array}
  \right]
\:\into\ 
  \kk\!
  \left[
  \begin{array}{@{\!}c@{\!\!}}\includegraphics[height=29mm]{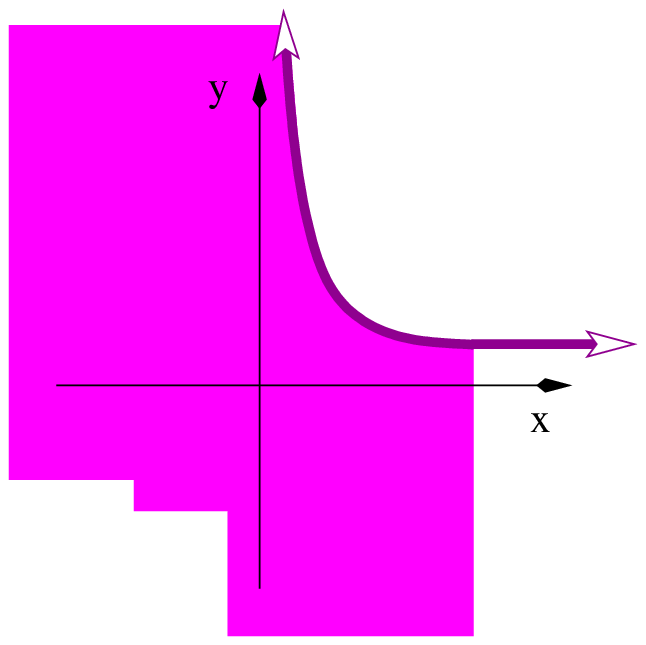}\end{array}
  \right]
\oplus\,
  \kk\!
  \left[
  \begin{array}{@{\!}c@{\!\!}}\includegraphics[height=29mm]{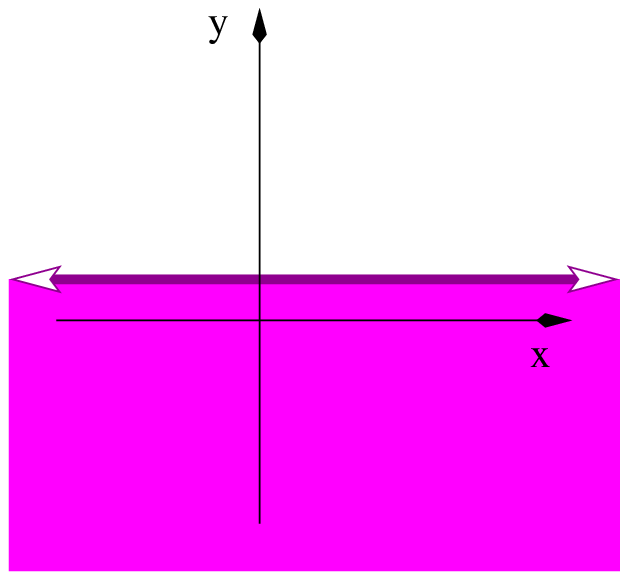}\end{array}
  \right]
\oplus\,
  \kk\!
  \left[
  \begin{array}{@{\!}c@{\!\!\!\!}}\includegraphics[height=29mm]{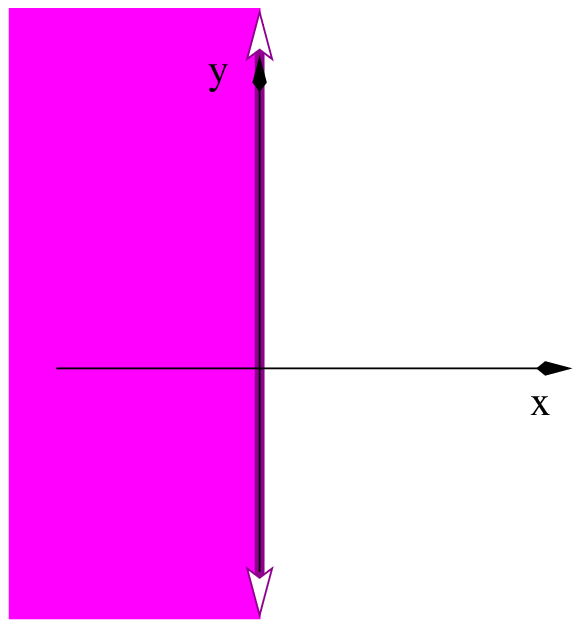}\end{array}
  \right]
$$
in which the global support along each face is extended downward so as
to become a quotient instead of a submodule of the original interval
module.
\end{example}

Primary decomposition requires some notion of finiteness, both for the
ambient context and for the module being decomposed.  In usual
commutative algebra, the noetherian condition serves both purposes.
Here, the closed polyhedral condition provides ambient finiteness, and
the following takes care of modules.

\begin{defn}\label{d:downset-hull}
A \emph{downset hull} of a module~$M$ over an arbitrary poset is an
injection $M \into \bigoplus_{j \in J} E_j$ with each $E_j$ being a
downset module.  The hull is \emph{finite} if $J$ is~finite.  The
module~$M$ is \emph{downset-finite} if it admits a finite downset
hull.
\end{defn}

\begin{remark}\label{r:tame}
The existence of primary decomposition in Theorem~\ref{t:primDecomp}
is intended for modules that are \emph{tame} \cite[Definitions~2.6
and~2.11]{hom-alg-poset-mods}.  Roughly speaking, each such module is
constant on finitely many regions that partition the poset.  However,
because primary decomposition deals only with essential submodules and
nothing akin to generators---in the pictures, only phenomena near the
upper boundary matter, not anything near the lower boundary---it only
requires the downset half of the tame condition.  In contrast, the
tame condition is equivalent (by the syzygy theorem for modules over
posets \cite[Theorem~6.12]{hom-alg-poset-mods}) to requiring a finite
fringe presentation \cite[Definition~3.16]{hom-alg-poset-mods}, which
entails a finite upset cover in addition to a finite downset hull.  In
general, if the monoid algebra $\kk[Q_+]$ is noetherian, then for
$Q$-modules,
$$%
  \text{noetherian} \implies \text{tame} \implies \text{downset-finite}.
$$
\end{remark}

\begin{example}\label{e:tame}
The implications in Remark~\ref{r:tame} are strict even when
$\kk[Q_+]$ is noetherian.  The upset module $\kk\{x^a y^b \mid a + b
\geq 0\}$ for the half-plane above the antidiagonal line is a tame but
not noetherian $\ZZ^2$-module.  For a downset-finite but not tame
$\ZZ^2$-module, take the submodule of $\kk[\ZZ^2] \oplus \kk[\ZZ^2]$
that has degree $\bigl[\twoline ab\bigr]$ component
$$%
\begin{cases}
\quad\ \ 0 & \text{below the antidiagonal line } a + b = 1,
\\
\quad\ \ \kk^2 & \text{above the antidiagonal line } a + b = 1,
\\
\mathrm{span}(\bigl[\twoline ab\bigr]) & \text{on the antidiagonal line } a + b = 1.
\end{cases}
$$
This module is visibly downset-finite---it is a submodule of a direct
sum of two copies of the downset module~$\kk[\ZZ^2]$---but it is not
tame \cite[Example~4.25]{hom-alg-poset-mods}.
\end{example}

\begin{thm}\label{t:primDecomp}
Every downset-finite module over a polyhedral partially ordered group
admits a primary decomposition.
\end{thm}
\begin{proof}
If $M \into \bigoplus_{j=1}^k E_j$ is a downset hull of the
module~$M$, and $E_j \into \bigoplus_{i=1}^\ell E_{ij}$ is a primary
decomposition for each~$j$ afforded by Corollary~\ref{c:PF} and
Example~\ref{e:prim-decomp-downset}, then let $E^\tau$ be the direct
sum of the downset modules $E_{ij}$ that are $\tau$-coprimary.  Set
$M^\tau = \ker(M \to E^\tau)$.  Then $M/M^\tau$ is coprimary,
being a submodule of a coprimary module.  Moreover, $M \to
\bigoplus_{\tau} M/M^\tau$ is injective because its kernel is the
same as the kernel of $M \to \bigoplus_{ij} E_{ij}$, which is a
composite of two injections and hence injective by construction.
Therefore $M \to \bigoplus_{\tau} M/M^\tau$ is a primary
decomposition.
\end{proof}

\begin{example}\label{e:circular-cone}
The finiteness of primary decomposition depends on the polyhedral
condition that posits finiteness of the number of faces of the
positive cone (Definition~\ref{d:face}).  When the positive cone has
infinitely many faces, such as the positive half~$Q_+$ of the right
circular cone $x^2 + y^2 \leq z^2$ in~$Q = \RR^3$, the $Q$-module
$$%
  \kk[\del Q_+] = \kk[Q_+]/\kk[Q_+^\circ]
$$
does not admit a finite primary decomposition.  The module $M =
\kk[\del Q_+]$ has a vector space of dimension~$1$ on the boundary of
the positive cone and~$0$ elsewhere.  Every face of the positive cone
must get its own summand $M/M_i$ in Definition~\ref{d:primDecomp'} for
the homomorphism $M \to \bigoplus_{i=1}^r M/M~ _i$ there to be
injective, and in that case the infinite number of faces would force
the direct sum to become a direct product.  This particular example,
with the right circular cone, works as well in the discrete partially
ordered group~$\ZZ^3$ because the circle has infinitely many
rational~points.
\end{example}

\enlargethispage*{3.5ex}
\vspace{-.6ex}
\addtocontents{toc}{\protect\setcounter{tocdepth}{2}}

\vspace{-1.5ex}


\begin{thebibliography}{CLR$^{+\!}$15}
\raggedbottom

\bibitem[Bir42]{birkhoff42}
Garrett Birkhoff, \emph{Lattice-ordered groups}, Annals of
  Math.~\textbf{43} (1942), 298--331.

\bibitem[Cli40]{clifford40}
Alfred H. Clifford, \emph{Partially ordered abelian groups}, Annals of
  Math.~\textbf{41} (1940), 465--473.

\bibitem[ES96]{eisenbud-sturmfels1996}
David Eisenbud and Bernd Sturmfels, \emph{Binomial ideals}, Duke
   Math.~J. \textbf{84} (1996),
   1--45.

\bibitem[Goo86]{goodearl86}
Kenneth R. Goodearl, \emph{Partially Ordered Abelian Groups with
  Interpolation}, Mathematical Surveys and Monographs, Vol.\,20,
  American Mathematical Society, Providence, 1986.


\bibitem[KM14]{mesoprimary}
Thomas Kahle and Ezra Miller, \emph{Decompositions of commutative
  monoid congruences and binomial ideals}, Alg.\ and Numb.\ Theory
  \textbf{8} (2014), no.~6, 1297--1364. \
  doi:10.2140/ ant.2014.8-6

\bibitem[KMO16]{soccular}
Thomas Kahle, Christopher O'Neill, and Ezra Miller, \emph{Irreducible
  decomposition of binomial ideals}, Compositio Math.\ \textbf{152}
  (2016), 1319--1332.  doi:10.1112/S0010437X16007272

\bibitem[KS18]{kashiwara-schapira2018}
Masaki Kashiwara and Pierre Schapira, \emph{Persistent homology and
  microlocal sheaf theory}, J. of Appl. and Comput.\ Topology
  \textbf{2}, no.\,1--2 (2018), 83--113.\ \
  \textsf{arXiv:math.AT/1705.00955v6}

\bibitem[KS19]{kashiwara-schapira2019}
Masaki Kashiwara and Pierre Schapira, \emph{Piecewise linear sheaves},
  to appear in International Math. Res. Notices (IMRN), 2019.\ \
  \textsf{arXiv: math.AG/1805.00349v3}

\bibitem[Les15]{lesnick-interleav2015}
Michael Lesnick, \emph{The theory of the interleaving distance on
  multidimensional persistence modules},
  Found.\ Comput.\ Math.~(2015), no.~15, 613--650.\ \
  doi:10.1007/s10208-015-9255-y

\bibitem[Mil15]{fruitFlyModuli}
Ezra Miller, \emph{Fruit flies and moduli: interactions between
  biology and mathematics}, Notices of the American Math.\ Society
  \textbf{62} (2015), no.\,10, 1178--1184.  doi:10.1090/noti1290\ \

\bibitem[Mil17]{qr-codes}
Ezra Miller, \emph{Data structures for real multiparameter persistence
  modules}, preprint, 2017.  \textsf{arXiv:math.AT/1709.08155}

\bibitem[Mil20a]{hom-alg-poset-mods}
Ezra Miller, \emph{Homological algebra of modules over posets},
  submitted, 2020.  \textsf{arXiv:math.AT/ 2008.00063}

\bibitem[Mil20b]{strat-conical}
Ezra Miller, \emph{Stratifications of real vector spaces from
  constructible sheaves with conical microsupport}, submitted, 2020.
  \textsf{arXiv:math.AT/2008.00091}

\bibitem[Mil20c]{essential-real}
Ezra Miller, \emph{Essential graded algebra over polynomial rings with
  real exponents}, submitted, 2020.  \textsf{arXiv:math.AC/2008.03819}

\bibitem[MS05]{cca}
Ezra Miller and Bernd Sturmfels, \emph{Combinatorial commutative
  algebra}, Graduate Texts in Mathematics, Vol.\,227, Springer-Verlag,
  New York, 2005.

\bibitem[MT20]{primary-distance}
Ezra Miller and Ashleigh Thomas, \emph{Persistence distances via
  primary decomposition}, draft.

\bibitem[Rie40]{riesz40}
Frigyes Riesz, \emph{Sur quelques notions fondamentales dans la th\'eorie
  g\'en\'erale des op\'erations lin\'eaires}, Annals of
  Math.~\textbf{41} (1940), 174--206.

\end{thebibliography}
\end{document}